\documentclass[12pt,twoside]{amsart}
\usepackage{amsmath, amsthm, amscd, amsfonts, amssymb, graphicx}
\usepackage{enumerate}
\usepackage[colorlinks=true,
linkcolor=blue,
urlcolor=black,
citecolor=red]{hyperref}
\usepackage{mathrsfs}
\newtheorem{theorem}{Theorem}[section]
\newtheorem{lemma}{Lemma}[section]
\newtheorem{remark}{Remark}[section]

\newtheorem{corollary}{Corollary}[section]

\newtheorem{proposition}{Proposition}[section]
\numberwithin{equation}{section}

\begin{document}
\title{Further properties of accretive matrices}
\author{Hamid Reza Moradi, Shigeru Furuichi and Mohammad Sababheh}
\maketitle
\begin{abstract}
To better understand the algebra $\mathcal{M}_n$ of all $n\times n$ complex matrices, we explore the class of accretive matrices. This class has received renowned attention in recent years due to its role in complementing those results known for positive definite matrices.
More precisely, we have several results that allow a better understanding of accretive matrices. Among many results, we present order-preserving results, Choi-Davis-type inequalities, mean-convex inequalities, sub-multiplicative results for the real part, and new bounds of the absolute value of accretive matrices. These results will be compared with the existing literature. In the end, we quickly pass through related entropy results for accretive matrices.
\end{abstract}
\pagestyle{myheadings}
\markboth{\centerline {}}
{\centerline {}}
\bigskip
\bigskip
\section{Introduction}
Let $\mathcal{M}_n$ be the class of all $n\times n$ complex matrices. Inequalities among elements of $\mathcal{M}_n$ has been an active research area due to its applications in various fields, not to mention its role in understanding the algebra $\mathcal{M}_n$.

However, order among elements in $\mathcal{M}_n$ is restricted to the so-called Hermitian matrices. A matrix $A\in\mathcal{M}_n$ is said to be Hermitian if $A^*=A$, where $A$ is the conjugate transpose of $A$. A special class of the Hermitian matrices is the positive ones. We recall that a matrix $A\in\mathcal{M}_n$ is said to be positive semi-definite, and we write $A\geq 0,$ if it satisfies $\left<Ax,x\right>\geq 0,$ for all $x\in\mathbb{C}^n$, where $\left<\cdot,\cdot\right>$ denotes the usual inner product in $\mathbb{C}^n$. The notation $\mathcal{M}_n^+$ will denote the class of positive semidefinite matrices in $\mathcal{M}_n$. Further, if $A\in\mathcal{M}_n^+$ is invertible, we say that $A$ is positive definite, and we write $A\in\mathcal{P}_n$ or $A>0$.
Having defined $\mathcal{M}_n^+$, a partial order on $\mathcal{H}_n$, the class of all Hermitian matrices in $\mathcal{M}_n$, can be defined. For $A,B\in\mathcal{H}_n$, we say that $A\leq B$ if $B-A\geq 0.$ If $B-A>0$, then we write $B>A$.

Defining this order on $\mathcal{H}_n$ then proposes the question about possible functional ordering in a way that simulates the field of real numbers. For example, if $f:J\to\mathbb{R}$ is an increasing function on the interval $J$ then $f(a)\leq f(b)$ for any $a,b\in J$ satisfying $a\leq b.$ The natural question then arises about the validity of the conclusion $f(A)\leq f(B)$ when $A,B\in\mathcal{H}_n$ are such that $A\leq B$. This turns out to be much more complicated.

For $A\in\mathcal{H}_n$, let $\sigma(A)$ denote the spectrum of $A$. An interval $J$ containing $\sigma(A)$ will be denoted as $J_{A}$. If $f:J_A\to \mathbb{R}$ is a given function, then $f(A)$ is defined via the simple identity $f(A)=U{\text{diag}}(f(\lambda_i))U^^*,$ where $U{\text{diag}}(\lambda_i)U^*$ is a spectral decomposition of $A$, in which $U$ is unitary and $\{\lambda_i:i=1,\cdots,n\}=\sigma(A).$

It is unfortunate that a monotone increasing function $f:J_{A,B}\to \mathbb{R}$ does not satisfy $f(A)\leq f(B)$ even when $A,B\in\mathcal{H}_n$ are such that $A\leq B$. This unpleasant scenario can be also said about convex functions.

This urges the search for possible classes of functions or matrices that could satisfy matrix inequalities as in the scalar case. For this, operator monotone functions were defined as those functions preserving order among Hermitian matrices. That is, a function $f:J\to\mathbb{R}$ is said to be operator monotone if $f(A)\leq f(B)$ for any $A,B\in\mathcal{H}_n$ are such that $A\leq B$ and $\sigma(A),\sigma(B)\subset J.$ Further, $f$ will be called operator convex if $f((1-t)A+tB)\leq (1-t)f(A)+tf(B)$ for all $t\in [0,1]$, where $A,B\in \mathcal{H}_n$ are such that $\sigma(A),\sigma(B)\subset J.$ If $-f$ is operator monotone, it is said to be operator monotone decreasing, and if $-f$ is operator convex it is said to be operator concave.

We refer the reader to \cite[Chapter V]{bhatia} for an excellent discussion of operator monotone and operator convex functions. We also refer the reader to \cite{1, choi, davis, book1, 10, kar, 5, n1, 12, n2, bel, hac} for a good list of references treating matrix orders.

In recent years, more interest has grown in studying inequalities among the so-called accretive matrices. Recall that a matrix $A\in\mathcal{M}_n$ is said to be accretive if ${\mathfrak R}(A)>0$, where $\mathfrak{R}(A)$ is the real part of $A$ defined by $\mathfrak{R}(A)=\frac{A+A^*}{2}.$ The class of accretive matrices in $\mathcal{M}_n$ will be denoted by $\Pi_n$. It is clear that $\mathcal{P}_n\subset \Pi_n$. Since elements of $\Pi_n$ are not Hermitian, the predefined order does not apply to $\Pi_n$. This is why inequalities among accretive matrices are usually stated in terms of their real parts. We must introduce sectorial matrices to deal with inequalities in $\Pi_n$. If $0\leq \alpha<\frac{\pi}{2},$ and if $A\in\mathcal{M}_n$ is such that
$$\{\left<Ax,x\right>:x\in\mathbb{C}^n,\|x\|=1\}\subset \{z\in\mathbb{C}:\mathfrak{R}(z)>0, |\mathfrak{I}(z)|\leq(\tan\alpha) \mathfrak{R}(z)\},$$ then $A$ will be called a sectorial matrix and we simply write $A\in\Pi_{n}^{\alpha}$, where $\mathfrak{I}(z)$ denotes the imaginary part of $z$. We refer the reader to \cite{1, bedann, bedlama, bedrim, drury,  johnson, Lin1, Mathias1992, rais, sano, Tan1} for an almost comprehensive overview of the progress that has been made in studying inequalities in $\Pi_n$. We emphasize here that whenever we use the notation $\Pi_n^{\alpha}$ in this paper, we implicitly understand that $0\leq \alpha<\frac{\pi}{2}.$ We also remark that  a matrix is accretive if and only if it is sectorial \cite{bedrim}.

The study of accretive matrices differs from that of Hermitian matrices because a partial order among members of $\Pi_n$ is not as well established as that in $\mathcal{H}_n$. So, in studying inequalities among members of $\Pi_n$, we usually refer to the real parts of these elements, noting that the real part of any matrix is in $\mathcal{H}_n$.

Our target in this paper is to study further possible inequalities among matrices in $\Pi_n$, where we extend some of the well-established inequalities in $\mathcal{P}_n$ or $\mathcal{M}_n^+$ to the class $\Pi_n$. For this to be done, we first need to define $f(A)$ where $A\in\Pi_n$ and $f:J_{A}\to\mathbb{R}.$

Given $A\in\mathcal{M}_n$, let $f:\mathcal{D}\to\mathbb{C}$ be a complex-valued function defined on a domain that contains $\sigma(A)$ in its interior. If $f$ is analytic in $\mathcal{D}$, we define
\begin{equation}\label{eq_int}
f(A)=\frac{1}{2\pi i}\int_{C}f(z)(zI-A)^{-1}dz,
\end{equation}
where $C$ is any simple closed curve in $\mathcal{D}$ that surrounds $\sigma(A)$. Practically, this generalizes the well-known complex Cauchy integral formula.

Now if $A\in\Pi_n$, then $\sigma(A)\cap (-\infty,0]=\varnothing.$ Therefore, if $f$ is analytic in any domain that avoids the negative $x-$axis, then $f(A)$ can be defined via \eqref{eq_int}.
For simplicity, we will use the notation
$$ \mathfrak{m}=\{f:(0,\infty)\to (0,\infty); f\; {\text{is a matrix monotone function with}} \;f(1)=1\}.$$
The following lemmas deserve mentioning here.
\begin{lemma}\label{hansen_lem} \cite[Theorem 4.9]{hansen1}
Let $f\in \mathfrak{m}.$ Then
$$f(x)=\int_{0}^{1}(1!_tx)d\nu_f(t),$$ where $\nu_f$ is a probability measure on $[0,1]$ and $1!_tx=(1-t+tx^{-1})^{-1}.$
\end{lemma}
\begin{lemma}\cite[Theorem V.4.7]{bhatia}
Let $f\in\mathfrak{m}$. Then $f$ has an analytic continuation to $\mathbb{C}\backslash (-\infty,0]$.
\end{lemma}
Thus, if $f\in\mathfrak{m},$ we may deal with its analytic continuation to find $f(A)$ for any matrix $A$ whose spectrum avoids the negative $x-$axis, where we can use \eqref{eq_int}.
Matrix monotone functions and matrix concave functions are strongly related, as follows \cite[Theorem 2.4]{ Uchiyama} and \cite[Theorems 2.1, 2.3, 3.1, 3.7]{ando_hiai}.
\begin{proposition}\label{oper_intro_prop}
Let $f:(0,\infty)\to [0,\infty)$ be continuous. Then
\begin{enumerate}
\item[(i)] $f$ is matrix monotone decreasing if and only if $f$ is matrix convex and $f(\infty)<\infty$.
\item[(ii)] $f$ is matrix monotone increasing if and only if $f$ is matrix concave.
\end{enumerate}
\end{proposition}
Consequently, $f\in\mathfrak{m}$ means that $f$ is operator monotone and operator concave.

On the other hand, the following two lemmas from \cite{1} will be needed in the sequel.
\begin{lemma}\label{lem_f1} Let $ f\in \mathfrak{m} $ and $ A\in\Pi_n$. Then
\begin{equation*}
\mathfrak R (f(A))\geq\ f(\mathfrak R A).
\end{equation*}
Consequently, if $A$ is accretive, then so is $f(A)$.
\end{lemma}
\begin{lemma}\label{lem_f2} Let $ f\in \mathfrak{m}$ and $ A\in\Pi_n^{\alpha}.$ Then
\begin{equation*}
\mathfrak R (f(A))\leq\ \sec^{2}(\alpha) \;f(\mathfrak R A).
\end{equation*}
\end{lemma}
We recall that a linear mapping $\Phi:\mathcal{M}_n\to\mathcal{M}_n$ is said to be positive if $\Phi(A)\in\mathcal{M}_n^+$ whenever $A\in\mathcal{M}_n^+$. Further, if $\Phi(I)=I$, then $\Phi$ is said to be a unital positive linear mapping. The celebrated Choi-Davis inequality states that \cite{Ando,choi}
\begin{align*}
\Phi(f(A))\leq f(\Phi(A)),
\end{align*}
for $f\in\mathfrak{m}$ and $A\in\mathcal{M}_n^+$, where $\Phi$ is a unital positive linear mapping. When $A$ is accretive, we have the following version of this inequality \cite[Theorem 7.1]{1}.
\begin{lemma}\label{choi_davis_accretive}
Let $ f\in \mathfrak{m} $ , $ \Phi $ be a unital positive linear map and $ A\in\Pi_n^{\alpha}$. Then
$$\mathfrak R f(\Phi(A))\geq\cos^2(\alpha)\;\mathfrak R\Phi(f(A)).$$
\end{lemma}
The so-called operator mean is strongly related to the class $\mathfrak{m}$. Given $A,B\in\mathcal{P}$ and $f\in\mathfrak{m},$ we define $\sigma_f:\mathcal{P}\times \mathcal{P}\to\mathcal{P}$ by
\begin{equation}\label{eq_sigma}
A\sigma_f B=A^{\frac{1}{2}}f\left(A^{-\frac{1}{2}}BA^{-\frac{1}{2}}\right)A^{\frac{1}{2}}.
\end{equation}
This binary operation is usually called operator mean, associated with $f$. If no confusion arises, we use $\sigma$ instead of $\sigma_f$. The theory of operator means has received considerable attention in the literature, as seen in \cite{ando_1,3,nis,pus}. The theory of operator means has been extended to accretive matrices in \cite{1}, using the same identity as in \eqref{eq_sigma}. We refer the reader to \cite{1} for a detailed discussion of this topic. We also refer the reader to \cite{drury,rais,sano} for interesting related discussion.

Extending some results from \cite{drury,rais}, the following inequality was shown in \cite{1} for any $A,B\in\Pi_n$ and any operator mean $\sigma$ (or $\sigma_f$ for some $f\in\mathfrak{m}$):
\begin{equation}\label{eq_ineq_sig}
\mathfrak{R}A\sigma\mathfrak{R}B\leq\mathfrak{R}(A\sigma B)\leq\sec^2\alpha\;\mathfrak{R}A\sigma\mathfrak{R}B.
\end{equation}
The following lemma has also been shown in \cite{1}.
\begin{lemma}\label{thm_amgmhm_accretive}
Let $ A, B\in\Pi_{n}^{\alpha}$ for some $0\leq \alpha<\frac{\pi}{2}$. If $f\in \mathfrak{m}$ is such that $f'(1)=t$ for some $t\in (0,1),$ then
\begin{equation*}
\cos^{2}(\alpha)\;\mathfrak R(A!_{t}B) \ \leq\ \mathfrak R(A\sigma_f B)\ \leq\ \sec^{2}(\alpha)\;\mathfrak R(A\nabla_{t}B),
\end{equation*}
where $A!_tB=((1-t)A^{-1}+tB^{-1})^{-1}$ and $A\nabla_tB=(1-t)A+tB$ are the weighted harmonic and arithmetic means, respectively.
\end{lemma}

The next section presents several new relations and inequalities for elements in $\Pi_n$ and $\Pi_n^{\alpha}.$ To make it easier for the reader to follow, we will emphasize the significance of each result by presenting the existing related result in the literature. Our discussion will include order-preserving inequalities, Choi-Davis-type inequalities, mean inequalities, entropy results, and other characterizations.

Among the most interesting findings in this paper, we show that if $A>0$, then
\[\mathfrak R\left( Y{{A}^{-1}}Y \right)\le \mathfrak RY\;{{A}^{-1}}\;\mathfrak RY,\]
for any $Y\in {{\mathcal M}_{n}}$. We also show that if $T\in\Pi_{n}^{\alpha},$ then
\[\left| T \right|\le \sec \alpha \left| {{\left( \mathfrak RT \right)}^{\frac{1}{2}}}U{{\left( \mathfrak RT \right)}^{\frac{1}{2}}} \right|\]
for some unitary $U$, where $|X|=(X^*X)^{1/2},$ when $X\in\mathcal{M}_n$.

Discussion of entropy-like results for accretive matrices will be presented, as a new track in this field.

Many other results about accretive matrices will be shown, and a glimpse of the relation with the literature will be presented to make it easier for the reader to comprehend the whole picture.
\section{Main Results}
In this section, we present our results. To make it easier and more accessible for the reader, we present these results in consequent subsections.
\subsection{Order preserving results}
We begin with a L\"{o}wner-Heinz theorem for accretive matrices. More precisely, if $f\in\mathfrak{m}$ and $0<A\leq B$, then $f(A)\leq f(B)$. This is indeed the definition of operator monotony. The following result extends this to the class of sectorial matrices by appealing to the real parts.
\begin{theorem}\label{1}
Let $f\in\mathfrak{m}$ and $A,B\in \Pi_n^{\alpha}$. Then
	\[\mathfrak R A\le \mathfrak R B \quad\Rightarrow \quad \mathfrak Rf\left( A \right)\le {{\sec }^{2}}\alpha \;\mathfrak Rf\left( B \right).\]
In particular,
\begin{equation}\label{h1}
\mathfrak RA\le \mathfrak RB\quad\Rightarrow \quad\mathfrak R{{A}^{r}}\le {{\sec }^{2}}\alpha \;\mathfrak R{{B}^{r}};\text{ }0\le r\le 1.
\end{equation}
\end{theorem}
\begin{proof}
We have
\[\begin{aligned}
\mathfrak Rf\left( A \right)&\le {{\sec }^{2}}\alpha \;f\left( \mathfrak RA \right) \quad \text{(by Lemma \ref{lem_f2})}\\
& \le {{\sec }^{2}}\alpha \;f\left( \mathfrak RB \right) \quad \text{(since $f$ is operator monotone and $\mathfrak{R}A\leq \mathfrak{R}B$) }\\
& \le {{\sec }^{2}}\alpha \;\mathfrak Rf\left( B \right)\quad \text{(by Lemma \ref{lem_f1})}.
\end{aligned}\]
This completes the proof.
\end{proof}
We know that if $A\in \Pi_n^{\alpha}$, then \cite{drury_lin}
\begin{equation}\label{6}
{{\left( \mathfrak RA \right)}^{-1}}\le {{\sec }^{2}}\alpha\; \mathfrak R{{A}^{-1}}.
\end{equation}
The following result is an application of the inequality \eqref{6},
\[\mathfrak RA\sharp\mathfrak R{{A}^{-1}}\ge \frac{1}{\sec \alpha }\left( \mathfrak RA\sharp{{\left( \mathfrak RA \right)}^{-1}} \right)=\frac{1}{\sec \alpha }.\]
Therefore,
\[\mathfrak RA\sharp \mathfrak R{{A}^{-1}}\ge \frac{1}{\sec \alpha }.\]
Here the notation $\sharp$ refers to the geometric mean, which is defined for any $A,B\in\Pi_n$ as follows
$$A\sharp B=A^{\frac{1}{2}}\left(A^{-\frac{1}{2}}BA^{-\frac{1}{2}}\right)^{\frac{1}{2}}A^{\frac{1}{2}}.$$
\begin{remark}
Notice that when $A,B\in\Pi_n^{\alpha}$, we have for $0\le r \le 1$
\[\mathfrak RA\le \mathfrak R B \text{ }\Rightarrow \text{ }\mathfrak R{{B}^{-r}}\le {{\sec }^{4}}\alpha\; \mathfrak R{{A}^{-r}},\]
since
\[\begin{aligned}
\mathfrak R{{B}^{-r}}&\le {{\left( \mathfrak R{{B}^{r}} \right)}^{-1}} \\
& \le {{\sec }^{2}}\alpha \;{{\left( \mathfrak R{{A}^{r}} \right)}^{-1}}\quad \text{(by \eqref{h1})} \\
& \le {{\sec }^{4}}\alpha \;\mathfrak R{{A}^{-r}} \quad \text{(by \eqref{6})}.
\end{aligned}\]
\end{remark}
\subsection{Choi-Davis type inequalities}
It is known that if $f\in\mathfrak{m}, A_i\in\mathcal{M}_n^+$ and $C_i\in\mathcal{M}_n$ are such that $\sum_{i=1}^{k}C_i^*C_i=I,$ then \cite[Theorem 1.9]{book1}
\begin{align}\label{eq_cc*}
f\left(\sum_{i=1}^{k}C_i^*A_iC_i\right)\geq \sum_{i=1}^{k}C_i^*f(A_i)C_i.
\end{align}
At this point, we show the accretive version of this inequality. We notice that when $A\in\Pi_n^{\alpha}$, then $C^*AC\in\Pi_n^{\alpha}$ for any $C\in\mathcal{M}_n$.
In order to show the accretive version of \eqref{eq_cc*}, we first present the following simple lemma.
\begin{lemma}
Let $A,B\in\Pi_n^{\alpha}.$ Then $A+B\in\Pi_n^{\alpha}.$
\end{lemma}
\begin{proof}
By definition of $\Pi_n^{\alpha},$ we have
\begin{align*}
|\mathfrak{I}\left<Ax,x\right>|\leq\tan\alpha\;\mathfrak{R}\left<Ax,x\right>\;{\text{and}}\;|\mathfrak{I}\left<Bx,x\right>|\leq\tan\alpha\;\mathfrak{R}\left<Bx,x\right>, x\in\mathbb{C}^n.
\end{align*}
Adding these two inequalities, we get
\begin{align*}
\tan\alpha\;\mathfrak{R}\left<(A+B)x,x\right>&\geq |\mathfrak{I}\left<Ax,x\right>|+|\mathfrak{I}\left<Bx,x\right>|\\
&\geq |\mathfrak{I}\left<Ax,x\right>+\mathfrak{I}\left<Bx,x\right>|\\
&=|\mathfrak{I}\left<(A+B)x,x\right>|.
\end{align*}
This completes the proof.
\end{proof}
Consequently, if $A_i\in\Pi_n^{\alpha}$ and $C_i\in\mathcal{M}_n$, ($i=1,\cdots,k$), then $\sum_{i=1}^{k}C_i^*A_iC_i\in\Pi_n^{\alpha}.$
We are ready to show the sectorial version of \eqref{eq_cc*}.
\begin{proposition}
Let $A_i\in\Pi_n^{\alpha}$ and $C_i\in\mathcal{M}_n$, ($i=1,\cdots,k$) be such that $\sum_{i=1}^{k}C_i^*C_i=I.$ Then
\begin{equation}\label{3}
\mathfrak R\left( \sum\limits_{i=1}^{n}{C_{i}^{*}f\left( {{A}_{i}} \right){{C}_{i}}} \right)\le {{\sec }^{2}}\alpha\; \mathfrak Rf\left( \sum\limits_{i=1}^{n}{C_{i}^{*}{{A}_{i}}{{C}_{i}}} \right),
\end{equation}
where $\sum\nolimits_{i=1}^{n}{C_{i}^{*}{{C}_{i}}}=I$.
\end{proposition}
\begin{proof}
First, we notice that if $C\in\mathcal{M}_n$ is such that $C^*C=I,$ then the mapping $\Phi:\mathcal{M}_n\to\mathcal{M}_n$ defined by $\Phi(X)=C^*XC$ is unital positive linear mapping. Therefore, Lemma \ref{choi_davis_accretive} implies
\begin{equation}\label{2}
\mathfrak R\left( {{C}^{*}}f\left( A \right)C \right)\le {{\sec }^{2}}\alpha\; \mathfrak Rf\left( {{C}^{*}}AC \right)
\end{equation}
where $f\in\mathfrak{m}, {{C}^{*}}C=I$ and $A\in\Pi_n^{\alpha}.$
Let
	\[X=\left[ \begin{matrix}
{{A}_{1}} & {} & {} & O \\
{} & {{A}_{2}} & {} & {} \\
{} & {} & \ddots & {} \\
O & {} & {} & {{A}_{k}} \\
\end{matrix} \right]\text{ and }\widetilde{C}=\left[ \begin{matrix}
{{C}_{1}} \\
{{C}_{2}} \\
\vdots \\
{{C}_{k}} \\
\end{matrix} \right].\]
It follows, by the same argument preceding the theorem, that ${{\widetilde{C}}^{*}}X\widetilde{C}\in\Pi_n^{\alpha}.$ Now, noting that ${{\widetilde{C}}^{*}}\widetilde{C}=I$, we have
	\[\begin{aligned}
\mathfrak R\left( \sum\limits_{i=1}^{n}{C_{i}^{*}f\left( {{A}_{i}} \right){{C}_{i}}} \right)&=\mathfrak R\left( {{\widetilde{C}}^{*}}f\left( X \right)\widetilde{C} \right) \\
& \le {{\sec }^{2}}\alpha \;\mathfrak Rf\left( {{\widetilde{C}}^{*}}X\widetilde{C} \right) \quad \text{(by \eqref{2})}\\
& ={{\sec }^{2}}\alpha \;\mathfrak Rf\left( \sum\limits_{i=1}^{n}{C_{i}^{*}{{A}_{i}}{{C}_{i}}} \right).
\end{aligned}\]
This completes the proof.
\end{proof}
Extending \eqref{2}, we can state the following result.
\begin{theorem}
Let $f\in\mathfrak{m}$ and let $C\in\mathcal{M}_n$ be such that ${{C}^{*}}C\le I$. Then
\[\mathfrak R\left( {{C}^{*}}f\left( A \right)C \right)\le {{\sec }^{2}}\alpha \; \mathfrak Rf\left( {{C}^{*}}AC \right),\]
for any $A\in\Pi_n^{\alpha}.$
\end{theorem}
\begin{proof}
Put $D=\sqrt{I-{{C}^{*}}C}$, where $C$ is a contraction (i.e., $C^*C\leq I$). Since ${{C}^{*}}C+{{D}^{*}}D=I$, we can write from \eqref{3} that
	\[\begin{aligned}
\mathfrak R\left( {{C}^{*}}f\left( A \right)C \right)&\le \mathfrak R\left( {{C}^{*}}f\left( A \right)C+{{D}^{*}}f\left( O \right)D \right) \\
& \le {{\sec }^{2}}\alpha\; \mathfrak Rf\left( {{C}^{*}}AC+{{D}^{*}}OD \right) \\
& = {{\sec }^{2}}\alpha \; \mathfrak Rf\left( {{C}^{*}}AC \right).
\end{aligned}\]
\end{proof}
\subsection{Means inequalities}
We have seen in \eqref{eq_ineq_sig} that
\[\mathfrak RA\sigma \mathfrak RB\le \mathfrak R\left( A\sigma B \right)\leq \sec^{2}\alpha\; \mathfrak RA\sigma \mathfrak RB.\]
In one way or another, this inequality is related to the so-called Callebaut inequality, whose matrix version states that if $A_i,B_i\in\mathcal{P}_n$, and if $\sigma$ is an operator mean, then \cite{mos}
\begin{align*}
\sum_{i=1}^{k}(A_i\sharp B_i)&\leq \left(\sum_{i=1}^{k}A_i\sigma B_i\right)\sharp \left(\sum_{i=1}^{k}A_i\sigma^{\perp} B_i\right)\leq \left(\sum_{i=1}^{k}A_i\right)\sharp\left(\sum_{i=1}^{k}B_i\right),
\end{align*}
where $\sigma^{\perp}$ is the operator mean associated with the function $\frac{t}{f(t)}.$ Here $f\in\mathfrak{m}$ is the function characterizing $\sigma$ as in \eqref{eq_sigma}.

Now we present the sectorial version of Callebaut inequality.
\begin{theorem}
Let $A_i,B_i\in\Pi_n^{\alpha}$ and let $\sigma=\sigma_f$ for some $f\in\mathfrak{m}.$ Then
\[\sum\limits_{i=1}^{n}{\mathfrak R{{A}_{i}}\sharp \mathfrak R{{B}_{i}}}\le \left( \sum\limits_{i=1}^{n}{\mathfrak R{{A}_{i}}\sigma \mathfrak R{{B}_{i}}} \right)\sharp\left( \sum\limits_{i=1}^{n}{\mathfrak R{{A}_{i}}{{\sigma }^{\bot }}\mathfrak R{{B}_{i}}} \right)\le {{\sec }^{2}}\alpha \; \mathfrak R\left( \sum\limits_{i=1}^{n}{{{A}_{i}}} \right)\sharp \mathfrak R\left( \sum\limits_{i=1}^{n}{{{B}_{i}}} \right).\]
\end{theorem}
\begin{proof}
We have
\[\begin{aligned}
\sum\limits_{i=1}^{n}{\mathfrak R{{A}_{i}}\sigma \mathfrak R{{B}_{i}}}&\le \left( \sum\limits_{i=1}^{n}{\mathfrak R{{A}_{i}}} \right)\sigma \left( \sum\limits_{i=1}^{n}{\mathfrak R{{B}_{i}}} \right) \\
& =\left( \mathfrak R\sum\limits_{i=1}^{n}{{{A}_{i}}} \right)\sigma \left( \mathfrak R\sum\limits_{i=1}^{n}{{{B}_{i}}} \right) \\
& \le \mathfrak R\left( \left( \sum\limits_{i=1}^{n}{{{A}_{i}}} \right)\sigma \left( \sum\limits_{i=1}^{n}{{{B}_{i}}} \right) \right) ,
\end{aligned}\]
where we have used \eqref{eq_ineq_sig} to obtain the last inequality. Thus, we have shown that
\[\sum\limits_{i=1}^{n}{\mathfrak R{{A}_{i}}\sigma \mathfrak R{{B}_{i}}}\le \mathfrak R\left( \left( \sum\limits_{i=1}^{n}{{{A}_{i}}} \right)\sigma \left( \sum\limits_{i=1}^{n}{{{B}_{i}}} \right) \right).\]
We also have
\[\sum\limits_{i=1}^{n}{\mathfrak R{{A}_{i}}{{\sigma }^{\bot }}\mathfrak R{{B}_{i}}}\le \mathfrak R\left( \left( \sum\limits_{i=1}^{n}{{{A}_{i}}} \right){{\sigma }^{\bot }}\left( \sum\limits_{i=1}^{n}{{{B}_{i}}} \right) \right).\]
In particular,
\[\sum\limits_{i=1}^{n}{\mathfrak R{{A}_{i}}\sharp \mathfrak R{{B}_{i}}}\le \mathfrak R\left( \left( \sum\limits_{i=1}^{n}{{{A}_{i}}} \right)\sharp \left( \sum\limits_{i=1}^{n}{{{B}_{i}}} \right) \right)\] because $\sharp^{\perp}=\sharp$.
Notice that
\[\begin{aligned}
& \left( \sum\limits_{i=1}^{n}{\mathfrak R{{A}_{i}}\sigma \mathfrak R{{B}_{i}}} \right)\sharp\left( \sum\limits_{i=1}^{n}{\mathfrak R{{A}_{i}}{{\sigma }^{\bot }}\mathfrak R{{B}_{i}}} \right) \\
& \le \mathfrak R\left( \left( \sum\limits_{i=1}^{n}{{{A}_{i}}} \right)\sigma \left( \sum\limits_{i=1}^{n}{{{B}_{i}}} \right) \right)\sharp \mathfrak R\left( \left( \sum\limits_{i=1}^{n}{{{A}_{i}}} \right){{\sigma }^{\bot }}\left( \sum\limits_{i=1}^{n}{{{B}_{i}}} \right) \right) \\
& \le {{\sec }^{2}}\alpha \;\left( \mathfrak R\left( \sum\limits_{i=1}^{n}{{{A}_{i}}} \right)\sigma \mathfrak R\left( \sum\limits_{i=1}^{n}{{{B}_{i}}} \right) \right)\sharp\left( \mathfrak R\left( \sum\limits_{i=1}^{n}{{{A}_{i}}} \right){{\sigma }^{\bot }}\mathfrak R\left( \sum\limits_{i=1}^{n}{{{B}_{i}}} \right) \right) \\
& ={{\sec }^{2}}\alpha \; \mathfrak R\left( \sum\limits_{i=1}^{n}{{{A}_{i}}} \right)\sharp \mathfrak R\left( \sum\limits_{i=1}^{n}{{{B}_{i}}} \right) ,
\end{aligned}\]
where we have used \eqref{eq_ineq_sig} to obtain the last inequality.
Further,
\[\begin{aligned}
\sum\limits_{i=1}^{n}{\mathfrak R{{A}_{i}}\sharp \mathfrak R{{B}_{i}}}&=\sum\limits_{i=1}^{n}{\left( \mathfrak R{{A}_{i}}\sigma \mathfrak R{{B}_{i}} \right)\sharp\left( \mathfrak R{{A}_{i}}{{\sigma }^{\bot }}\mathfrak R{{B}_{i}} \right)} \\
& \le \left( \sum\limits_{i=1}^{n}{\mathfrak R{{A}_{i}}\sigma \mathfrak R{{B}_{i}}} \right)\sharp\left( \sum\limits_{i=1}^{n}{\mathfrak R{{A}_{i}}{{\sigma }^{\bot }}\mathfrak R{{B}_{i}}} \right).
\end{aligned}\]
Thus,
\[\sum\limits_{i=1}^{n}{\mathfrak R{{A}_{i}}\sharp \mathfrak R{{B}_{i}}}\le \left( \sum\limits_{i=1}^{n}{\mathfrak R{{A}_{i}}\sigma \mathfrak R{{B}_{i}}} \right)\sharp\left( \sum\limits_{i=1}^{n}{\mathfrak R{{A}_{i}}{{\sigma }^{\bot }}\mathfrak R{{B}_{i}}} \right)\le {{\sec }^{2}}\alpha \; \mathfrak R\left( \sum\limits_{i=1}^{n}{{{A}_{i}}} \right)\sharp \mathfrak R\left( \sum\limits_{i=1}^{n}{{{B}_{i}}} \right),\]
which completes the proof.
\end{proof}
Another mean-convex inequality can be stated as follows.
\begin{theorem}
Let $A,B,C,D\in \Pi_n^{\alpha}.$ Then
\[\mathfrak R\left( \lambda \left( A{{\sharp}_{t}}C \right)+\left( 1-\lambda \right)\left( B{{\sharp}_{t}}D \right) \right)\le {{\sec }^{2}}\alpha \left( \mathfrak R\left( \lambda A+\left( 1-\lambda \right)B \right){{\sharp}_{t}}\mathfrak R\left( \lambda C+\left( 1-\lambda \right)D \right) \right)\]
for any $0\le t,\lambda \le 1$.
\end{theorem}
\begin{proof}
Noting \eqref{eq_ineq_sig} and implementing basic properties of means, we have
\[\begin{aligned}
\lambda \mathfrak R\left( A{{\sharp}_{t}}C \right)+\left( 1-\lambda \right)\mathfrak R\left( B{{\sharp}_{t}}D \right)&\le {{\sec }^{2}}\alpha \left( \lambda \left( \mathfrak RA{{\sharp}_{t}}\mathfrak RC \right)+\left( 1-\lambda \right)\left( \mathfrak RB{{\sharp}_{t}}\mathfrak RD \right) \right) \\
& \le {{\sec }^{2}}\alpha \left( \left( \lambda \mathfrak RA+\left( 1-\lambda \right)\mathfrak RB \right){{\sharp}_{t}}\left( \lambda \mathfrak RC+\left( 1-\lambda \right)\mathfrak RD \right) \right) \\
& ={{\sec }^{2}}\alpha \left( \mathfrak R\left( \lambda A+\left( 1-\lambda \right)B \right){{\sharp}_{t}}\mathfrak R\left( \lambda C+\left( 1-\lambda \right)D \right) \right).
\end{aligned}\]
This completes the proof.
\end{proof}
\subsection{A sub-multiplicative result for the real part}
We have seen that the real part plays a key role in studying accretive matrices. This is due to the ability to compare Hermitian matrices only. Defining $\Phi:\mathcal{M}_n\to\mathcal{M}_n$ by $\Phi(X)=\mathfrak{R}X,$ we immediately see that $\Phi$ is a unital positive linear mapping. It is well known that for any such $\Phi$ and any $A,B\in\mathcal{P}_n$, one has \cite[Theorem 1.19]{book1}
\begin{equation*}\label{eq_supmultphi}
\Phi(B)\Phi(A)^{-1}\Phi(B)\leq \Phi(BA^{-1}B).
\end{equation*}
Notice that when $\Phi=\mathfrak{R},$ this inequality becomes an identity because of $A,B\in\mathcal{P}_n$.
Interestingly, this inequality can be extended to the following form: one matrix is positive definite, but the other is arbitrary.
\begin{theorem}
Let $A>0$. Then
	\[\mathfrak R\left( Y{{A}^{-1}}Y \right)\le \mathfrak RY\;{{A}^{-1}}\;\mathfrak RY,\]
for any $Y\in {{\mathcal M}_{n}}$.
\end{theorem}
\begin{proof}
We can see that for any $X\in\mathcal{M}_n$ and positive definite $A$,
\begin{equation}\label{4}
\mathfrak RX={{A}^{-\frac{1}{2}}}\;\mathfrak R\left( {{A}^{\frac{1}{2}}}X{{A}^{\frac{1}{2}}} \right)\;{{A}^{-\frac{1}{2}}}.
\end{equation}
Noting that $\mathfrak{R}X^2=(\mathfrak{R}X)^2-(\mathfrak{I}X)^2\leq (\mathfrak{R}X)^2,$ and letting $X={{A}^{-\frac{1}{2}}}Y{{A}^{-\frac{1}{2}}}$, we have
	\[\mathfrak R{{A}^{-\frac{1}{2}}}Y{{A}^{-1}}Y{{A}^{-\frac{1}{2}}}\le {{\left( \mathfrak R{{A}^{-\frac{1}{2}}}Y{{A}^{-\frac{1}{2}}} \right)}^{2}}.\]
Using \eqref{4} now, we have
\begin{equation}\label{5}
{{A}^{-\frac{1}{2}}}\;\mathfrak R\left( Y{{A}^{-1}}Y \right)\;{{A}^{-\frac{1}{2}}}\le {{A}^{-\frac{1}{2}}}\;\mathfrak RY\;{{A}^{-1}}\;\mathfrak RY\;{{A}^{-\frac{1}{2}}}.
\end{equation}
Multiplying both sides of \eqref{5} by ${{A}^{\frac{1}{2}}}$, we reach the desired result.
\end{proof}
\subsection{On the absolute value of accretive matrices}
If $X\in\Pi_n^{\alpha},$ and $0\le r\le 1$, then
\[\begin{aligned}
\frac{1}{{{\sec }^{2}}\alpha }\mathfrak R{{X}^{2r}}&\le {{\left( \mathfrak R{{X}^{2}} \right)}^{r}} \quad \text{(by Lemma \ref{lem_f2})}\\
& \le {{\left( \mathfrak RX \right)}^{2r}},
\end{aligned}\]
where the second inequality follows from the facts that $\mathfrak{R}X^2\leq (\mathfrak{R}X)^2$ for any $X\in\mathcal{M}_n$ and that $f(t)=t^r$ is operator monotone when $0\leq r\leq 1.$

Thus, we have shown that if $X\in\Pi_{n}^{\alpha}$, one has
\[\mathfrak R{{X}^{2r}}\le {{\sec }^{2}}\alpha \;{{\left( \mathfrak RX \right)}^{2r}};\text{ }0\le r\le 1.\]
In \cite{7}, it has been shown that
\begin{equation}\label{eq_sec_comp}
\left\| T \right\|\le \sec \alpha \;\left\| \mathfrak RT \right\|.
\end{equation}
In this subsection, we present refinements and further related results. More precisely, we show better bounds for $|T|$ rather than $\|T\|.$ First, we have the following basic lemmas.
\begin{lemma}\label{011_ppt}
\cite[Lemma 1]{x2} Let $A,B,C\in \mathcal{M}_n$ be such that $A,B\geq 0$. Then
\[\left[ \begin{matrix}
A & C \\
{{C}^{*}} & B \\
\end{matrix} \right]\ge 0\Leftrightarrow {{\left| \left\langle Cx,y \right\rangle \right|}^{2}}\le \left\langle Ax,x \right\rangle \left\langle By,y \right\rangle,\forall x,y\in\mathbb{C}^n.\]
\end{lemma}
\begin{lemma}\label{lem_ned_ppt}
\cite[Proposition 1.3.2]{x3} Let $A,B\geq 0$. Then $\left[ \begin{matrix}
A & X \\
{{X}^{*}} & B \\
\end{matrix} \right]\geq 0$ if and only if $X={{A}^{\frac{1}{2}}}K{{B}^{\frac{1}{2}}}$ for some contraction $K$.
\end{lemma}
\begin{lemma}\label{lem_nd_ppt2}
\cite{x7} Let $T\geq O$. Then for any vectors $x,y\in {{\mathbb{C}}^{n}}$,
\[\left| \left\langle Tx,y \right\rangle \right|\le \frac{\left\| T \right\|}{2}\left( \left| \left\langle x,y \right\rangle \right|+\left\| y \right\|\left\| x \right\| \right).\]
\end{lemma}
Now we show the following preliminary result, which we will need.
\begin{proposition}\label{prop_ppt}
Let $A\in {{\mathcal M}_{n}}$ with the polar decomposition $A=U\left| A \right|$. Then for any vectors $x,y\in {{\mathbb{C}}^{n}}$,
\[\left| \left\langle Ax,y \right\rangle \right|\le \frac{\left\| A \right\|}{2}\left( \left| \left\langle x,{{U}^{*}}y \right\rangle \right|+\left\| {{U}^{*}}y \right\|\left\| x \right\| \right).\]
\end{proposition}
\begin{proof}
Lemma \ref{lem_nd_ppt2} gives
\begin{equation}\label{7_ppt}
\left| \left\langle \left| A \right|x,y \right\rangle \right|\le \frac{\left\|\; \left| A \right|\; \right\|}{2}\left( \left| \left\langle x,y \right\rangle \right|+\left\| y \right\|\left\| x \right\| \right)=\frac{\left\| A \right\|}{2}\left( \left| \left\langle x,y \right\rangle \right|+\left\| y \right\|\left\| x \right\| \right)
\end{equation}
for any $A\in {{\mathbb M}_{n}}$ .
Assume that $A=U\left| A \right|$ be the polar decomposition of $A$. If we replace $y$ by ${{U}^{*}}y$, in the inequality \eqref{7_ppt}, we get
	\[\begin{aligned}
\left| \left\langle Ax,y \right\rangle \right|&=\left| \left\langle U\left| A \right|x,y \right\rangle \right| \\
& =\left| \left\langle \left| A \right|x,{{U}^{*}}y \right\rangle \right| \\
& \le \frac{\left\| A \right\|}{2}\left( \left| \left\langle x,{{U}^{*}}y \right\rangle \right|+\left\| {{U}^{*}}y \right\|\left\| x \right\| \right),
\end{aligned}\]
from which the required result follows.
\end{proof}
The next theorem will be the key tool to obtain our result about a possible bound of $T$, where $T\in\Pi_n^{\alpha}.$
\begin{theorem}\label{thm_ppt}
Let $A,X,B\in {{\mathcal M}_{n}}$. Then $\left[ \begin{matrix}
A & X \\
{{X}^{*}} & B \\
\end{matrix} \right]\geq 0$, if and only if for any vectors $x,y\in {{\mathbb{C}}^{n}}$,
\[\left| \left\langle Xx,y \right\rangle \right|\le \frac{1}{2}\left( \left| \left\langle {{A}^{\frac{1}{2}}}U{{B}^{\frac{1}{2}}}x,y \right\rangle \right|+\sqrt{\left\langle Ay,y \right\rangle \left\langle Bx,x \right\rangle } \right)\]
for some unitary $U$.
\end{theorem}
\begin{proof}
Suppose that $\left[ \begin{matrix}
A & X \\
{{X}^{*}} & B \\
\end{matrix} \right]\geq 0$. So,
\[\begin{aligned}
\left| \left\langle Xx,y \right\rangle \right|&=\left| \left\langle {{A}^{\frac{1}{2}}}K{{B}^{\frac{1}{2}}}x,y \right\rangle \right| \quad \text{(by Lemma \ref{lem_ned_ppt})}\\
& \le \frac{\left\| K \right\|}{2}\left( \left| \left\langle {{A}^{\frac{1}{2}}}U{{B}^{\frac{1}{2}}}x,y \right\rangle \right|+\left\| {{U}^{*}}{{A}^{\frac{1}{2}}}y \right\|\left\| {{B}^{\frac{1}{2}}}x \right\| \right) \quad \text{(by Proposition \ref{prop_ppt})}\\
& \le \frac{1}{2}\left( \left| \left\langle {{A}^{\frac{1}{2}}}U{{B}^{\frac{1}{2}}}x,y \right\rangle \right|+\left\| {{U}^{*}}{{A}^{\frac{1}{2}}}y \right\|\left\| {{B}^{\frac{1}{2}}}x \right\| \right) \quad \text{(since $\left\| K \right\|\le 1$)}\\
& =\frac{1}{2}\left( \left| \left\langle {{A}^{\frac{1}{2}}}U{{B}^{\frac{1}{2}}}x,y \right\rangle \right|+\sqrt{\left\langle {{A}^{\frac{1}{2}}}U{{U}^{*}}{{A}^{\frac{1}{2}}}y,y \right\rangle \left\langle Bx,x \right\rangle } \right) \quad \text{(since $U{{U}^{*}}=I$)}\\
& =\frac{1}{2}\left( \left| \left\langle {{A}^{\frac{1}{2}}}U{{B}^{\frac{1}{2}}}x,y \right\rangle \right|+\sqrt{\left\langle Ay,y \right\rangle \left\langle Bx,x \right\rangle } \right)
\end{aligned}\]
for any vectors $x,y\in {{\mathbb{C}}^{n}}$.
For the other side, if for any $x,y\in {{\mathbb{C}}^{n}}$,
	\[\left| \left\langle Xx,y \right\rangle \right|\le \frac{1}{2}\left(\left| \left\langle {{A}^{\frac{1}{2}}}U{{B}^{\frac{1}{2}}}x,y \right\rangle \right|+\sqrt{\left\langle Ay,y \right\rangle \left\langle Bx,x \right\rangle }\right)\]
holds, for some unitary $U$, then by the Cauchy-Schwarz inequality, we have
\begin{equation*}\label{9}
\begin{aligned}
\left| \left\langle {{A}^{\frac{1}{2}}}U{{B}^{\frac{1}{2}}}x,y \right\rangle \right|&=\left| \left\langle U{{B}^{\frac{1}{2}}}x,{{A}^{\frac{1}{2}}}y \right\rangle \right| \\
& \le \left\| U{{B}^{\frac{1}{2}}}x \right\|\left\| {{A}^{\frac{1}{2}}}y \right\| \\
& =\sqrt{\left\langle U{{B}^{\frac{1}{2}}}x,U{{B}^{\frac{1}{2}}}x \right\rangle \left\langle {{A}^{\frac{1}{2}}}y,{{A}^{\frac{1}{2}}}y \right\rangle } \\
& =\sqrt{\left\langle {{B}^{\frac{1}{2}}}{{U}^{*}}U{{B}^{\frac{1}{2}}}x,x \right\rangle \left\langle Ay,y \right\rangle } \quad \text{(since ${{U}^{*}}U=I$)}\\
& =\sqrt{\left\langle Bx,x \right\rangle \left\langle Ay,y \right\rangle }.
\end{aligned}
\end{equation*}
Therefore,
	\[\left| \left\langle Xx,y \right\rangle \right|\le \sqrt{\left\langle Ay,y \right\rangle \left\langle Bx,x \right\rangle }.\]
Now, the result follows by Lemma \ref{011_ppt}.
\end{proof}
Theorem \ref{thm_ppt} can be used to obtain the following bound of the inner product of accretive matrices, which entails a refinement of \eqref{eq_sec_comp}.
\begin{corollary}
Let $T\in\Pi_{n}^{\alpha},$ and let $x,y\in\mathbb{C}^n$ be arbitrary vectors. Then
\[\left| \left\langle Tx,y \right\rangle \right|\le \frac{\sec \alpha }{2}\left( \left| \left\langle {{\left( \mathfrak RT \right)}^{\frac{1}{2}}}U{{\left( \mathfrak RT \right)}^{\frac{1}{2}}}x,y \right\rangle \right|+\sqrt{\left\langle \mathfrak RTy,y \right\rangle \left\langle \mathfrak RTx,x \right\rangle } \right)\]
for some unitary matrix $U\in\mathcal{M}_n.$
In particular,
\begin{equation}\label{eq_ned_comp_cor}
\|T\|\leq \frac{\sec \alpha }{2}\left( r\left( U\mathfrak RT \right)+\left\| \mathfrak RT \right\| \right),
\end{equation}
where $r(\cdot)$ is the spectral radius.
\end{corollary}
\begin{proof}
It has been shown in \cite[Theorem 2.2]{6} that if $T\in\Pi_{n}^{\alpha}$, then
\begin{equation}\label{8}
\left[ \begin{matrix}
\sec \alpha \mathfrak RT & T \\
T & \sec \alpha \mathfrak RT \\
\end{matrix} \right]\ge 0.
\end{equation}
Now Theorem \ref{thm_ppt} implies the first desired inequality. For the second inequality, take the supremum over all unit vectors $x,y$ in the first inequality to get
\begin{equation}\label{9}
\begin{aligned}
\left\| T \right\|&\le \frac{\sec \alpha }{2}\left( \left\| {{\left( \mathfrak RT \right)}^{\frac{1}{2}}}U{{\left( \mathfrak RT \right)}^{\frac{1}{2}}} \right\|+\left\| \mathfrak RT \right\| \right) \\
& =\frac{\sec \alpha }{2}\left( r\left( {{\left( \mathfrak RT \right)}^{\frac{1}{2}}}U{{\left( \mathfrak RT \right)}^{\frac{1}{2}}} \right)+\left\| \mathfrak RT \right\| \right) \\
& =\frac{\sec \alpha }{2}\left( r\left( U\mathfrak RT \right)+\left\| \mathfrak RT \right\| \right).
\end{aligned}
\end{equation}
This completes the proof.
\end{proof}
\begin{remark}
The inequality \eqref{9} is a refinement of \eqref{eq_sec_comp}, since
\[\begin{aligned}
r\left( U\mathfrak RT \right) & \le \left\| U\mathfrak RT \right\| = \left\| \mathfrak RT \right\|.
\end{aligned}\]
\end{remark}
Now we show the main result in this subsection.
\begin{theorem}\label{7}
Let $T\in \Pi _{n}^{\alpha }$. Then
\[\left| T \right|\le \sec \alpha \left| {{\left( \mathfrak RT \right)}^{\frac{1}{2}}}U{{\left( \mathfrak RT \right)}^{\frac{1}{2}}} \right|\]
for some unitary $U$.
\end{theorem}
\begin{proof}
We know that if $\left[ \begin{matrix}
A & X \\
{{X}^{*}} & B \\
\end{matrix} \right]\ge 0$, then \cite{5}
	\[{{X}^{*}}X\le {{B}^{\frac{1}{2}}}{{U}^{*}}AU{{B}^{\frac{1}{2}}},\text{ for some unitary }U.\]
Here we present another proof of this result using a different method. It has been shown in \cite[Theorem 7]{FUJII} that $\left[ \begin{matrix}
A & X \\
{{X}^{*}} & B \\
\end{matrix} \right]\ge 0$ if and only if there is an operator $C$ such that $X={{C}^{*}}{{B}^{\frac{1}{2}}}$ and ${{C}^{*}}C\le A$. We can write
	\[{{X}^{*}}X={{\left| X \right|}^{2}}={{B}^{\frac{1}{2}}}C{{C}^{*}}{{B}^{\frac{1}{2}}}={{B}^{\frac{1}{2}}}{{\left| {{C}^{*}} \right|}^{2}}{{B}^{\frac{1}{2}}}.\]
Thus,
	\[{{B}^{-\frac{1}{2}}}{{\left| X \right|}^{2}}{{B}^{-\frac{1}{2}}}={{\left| {{C}^{*}} \right|}^{2}}.\]
Let $C=V\left| C \right|$ is the polar decomposition of $C$. We have
	\[{{V}^{*}}\left( {{B}^{-\frac{1}{2}}}{{\left| X \right|}^{2}}{{B}^{-\frac{1}{2}}} \right)V={{V}^{*}}{{\left| {{C}^{*}} \right|}^{2}}V={{\left| C \right|}^{2}}={{C}^{*}}C.\]
Therefore, by the assumption,
	\[{{V}^{*}}\left( {{B}^{-\frac{1}{2}}}{{\left| X \right|}^{2}}{{B}^{-\frac{1}{2}}} \right)V\le A.\]
So,
	\[{{\left| X \right|}^{2}}\le {{B}^{\frac{1}{2}}}\left( VA{{V}^{*}} \right){{B}^{\frac{1}{2}}}=\left( {{B}^{\frac{1}{2}}}V{{A}^{\frac{1}{2}}} \right)\left( {{A}^{\frac{1}{2}}}{{V}^{*}}{{B}^{\frac{1}{2}}} \right)={{\left| {{A}^{\frac{1}{2}}}{{V}^{*}}{{B}^{\frac{1}{2}}} \right|}^{2}}.\]
The proof is complete by assigning ${{V}^{*}}$ to new unitary $U$. Consequently, we showed that
\[{{\left| X \right|}^{2}}\le {{\left| {{A}^{\frac{1}{2}}}U{{B}^{\frac{1}{2}}} \right|}^{2}}.\]
Since the function $f\left( t \right)=\sqrt{t}$ is operator monotone, we get
\[\left| X \right|\le \left| {{A}^{\frac{1}{2}}}U{{B}^{\frac{1}{2}}} \right|.\]
By \eqref{8},
	\[\left[ \begin{matrix}
\sec \alpha \mathfrak RT & T \\
{{T}^{*}} & \sec \alpha \mathfrak RT \\
\end{matrix} \right]\ge 0.\]
Combining the two inequalities above, we get the desired result.
\end{proof}
\begin{remark}
It follows from Theorem \ref{7} that
\[\begin{aligned}
\left\| T \right\|&=\left\|\; \left| T \right|\; \right\| \\
& \le \sec \alpha \;\left\| \;\left| {{\left( \mathfrak RT \right)}^{\frac{1}{2}}}U{{\left( \mathfrak RT \right)}^{\frac{1}{2}}} \right| \;\right\| \\
& =\sec \alpha \;\left\| {{\left( \mathfrak RT \right)}^{\frac{1}{2}}}U{{\left( \mathfrak RT \right)}^{\frac{1}{2}}} \right\| \\
& =\sec \alpha \; r\left( U\mathfrak RT \right).
\end{aligned}\]
Therefore,
\[\left\| T \right\|\le \sec \alpha \; r\left( U\mathfrak RT \right);\]
which is a significant refinement of \eqref{eq_ned_comp_cor} and \eqref{eq_sec_comp}.
\end{remark}
\section{On the difference of two perspectives}
Let $\sigma_f$ be a matrix mean related to matrix monotone function $f\in\mathfrak{m}$. Then $A\sigma_f B:=A^{1/2}f\left(A^{-1/2}BA^{-1/2}\right)A^{1/2}$ is often sometimes called a perspective \cite{3}. It is not hard to check that the function $\ln_t(x):=\dfrac{x^t-1}{t}$ defined on $x>0$ with $0<t\le 1$, is matrix monotone.
Tsallis relative operator entropy is defined as
\begin{equation*}\label{definition00}
T_t(A|B):=A\sigma_{\ln_t}B=A^{1/2}\ln_t\left(A^{-1/2}BA^{-1/2}\right)A^{1/2}=\dfrac{A\sharp_tB-A}{t}.
\end{equation*}
In \cite{2}, the mathematical properties of $T_t(A|B)$ as the Tsallis relative operator entropy were studied.

We may define a difference between two perspectives as
\begin{equation*}\label{definition02}
{{D}_{f,g}}\left( A|B \right)=A{{\sigma }_{f}}B-A{{\sigma }_{g}}B,
\end{equation*}
for $f,g\in\mathfrak{m}$.
Here we mention some examples.
\begin{itemize}
\item[(i)] If we take $f(x):=(1-t)+tx$ and $g(x)=x^t$ for $t\in[0,1]$, then
$D_{f,g}(A|B)=A\nabla_t B-A\sharp_tB$, where $\nabla_t$ and $\sharp_t$ are the means associated with $f$ and $g$ respectively.
\item[(ii)] If we take $f(x):=\dfrac{x^t-1}{t}+1$, $g(x):=1$, then we get
$D_{f,g}(A|B)=T_t(A|B)$ the Tsallis relative operator entropy.
In addition, if we take $f(x):=\log x+1$, $g(x):=1$, then we get
$D_{f,g}(A|B)=S(A|B)$ the relative operator entropy.
\item[(iii)]If we take $f(x):=\dfrac{x^t-1}{t}+1$, $g(x):=\log x+1$, then
$D_{f,g}(A|B)=T_t(A|B)-S(A|B)$, which gives the difference between the Tsallis relative operator entropy and the relative operator entropy. And it is known that $S(A|B)\le T_t(A|B)$ for $0<t\le 1$.
\end{itemize}
In this section, we study $D_{f,g}(A|B)$ for accretive matrices $A,B$; as a new track in this research field.
\begin{theorem}
Let $A,B\in \Pi _{n}^{\alpha }$ and $f,g\in\mathfrak{m}$. Then for any invertible $C\in {{\mathcal M}_{n}}$,
\[{{C}^{*}}{{D}_{f,g}}\left( A| B \right)C={{D}_{f,g}}\left( {{C}^{*}}AC| {{C}^{*}}BC \right).\]
\end{theorem}
\begin{proof}
We have
\[\begin{aligned}
{{C}^{*}}{{D}_{f,g}}\left( A|B \right)C&={{C}^{*}}\left( {A{{\sigma }_{f}}B-A{{\sigma }_{g}}B} \right)C \\
& ={{{C}^{*}}\left( A{{\sigma }_{f}}B \right)C-{{C}^{*}}\left( A{{\sigma }_{g}}B \right)C} \\
& ={{{C}^{*}}\left( A{{\sigma }_{f}}B \right)C-{{C}^{*}}\left( A{{\sigma }_{g}}B \right)C} \\
& ={{{C}^{*}}AC{{\sigma }_{f}}{{C}^{*}}BC-{{C}^{*}}AC{{\sigma }_{g}}{{C}^{*}}BC} \\
& ={{D}_{f,g}}\left( {{C}^{*}}AC|{{C}^{*}}BC \right).
\end{aligned}\]
\end{proof}
With the inclusion of real parts of sectorial matrices, one may obtain further bounds as follows.
\begin{theorem}\label{thm_entropy_1}
Let $A,B\in \Pi _{n}^{\alpha }$ and $f,g\in\mathfrak{m}$. Then
\begin{eqnarray*}
&& {{D}_{f,g}}\left( \mathfrak RA\text{ }\!\!|\!\!\text{ }\mathfrak RB \right)+\left({1-{{\sec }^{2}}\alpha }\right)\left( \mathfrak RA{{\sigma }_{g}}\mathfrak RB \right)\\
&& \le \mathfrak R\left( {{D}_{f,g}}\left( A\text{ }\!\!|\!\!\text{ }B \right) \right)\\
&& \le
{{D}_{f,g}}\left( \mathfrak RA| \mathfrak RB \right)+\left({{{\sec }^{2}}\alpha -1}\right)\left( \mathfrak RA{{\sigma }_{f}}\mathfrak RB \right).
\end{eqnarray*}
\end{theorem}
\begin{proof}
We have
\[\begin{aligned}
\mathfrak R\left( {{D}_{f,g}}\left( A| B \right) \right)&=\mathfrak R\left( {A{{\sigma }_{f}}B-A{{\sigma }_{g}}B} \right) \\
& ={\mathfrak R\left( A{{\sigma }_{f}}B \right)-\mathfrak R\left( A{{\sigma }_{g}}B \right)} \\
& \ge {\mathfrak RA{{\sigma }_{f}}\mathfrak RB-{{\sec }^{2}}\alpha \left( \mathfrak RA{{\sigma }_{g}}\mathfrak RB \right)} \\
& ={\mathfrak RA{{\sigma }_{f}}\mathfrak RB-\mathfrak RA{{\sigma }_{g}}\mathfrak RB}+\left({1-{{\sec }^{2}}\alpha }\right)\left( \mathfrak RA{{\sigma }_{g}}\mathfrak RB \right) \\
& ={{D}_{f,g}}\left( \mathfrak RA|\mathfrak RB \right)+\left({1-{{\sec }^{2}}\alpha }\right)\left( \mathfrak RA{{\sigma }_{g}}\mathfrak RB \right),
\end{aligned}\]
where we have used \eqref{eq_ineq_sig} to obtain the first inequality in these computations.
Noting \eqref{eq_ineq_sig}, we also have
\[\begin{aligned}
\mathfrak R\left( {{D}_{f,g}}\left( A| B \right) \right)&=\mathfrak R\left( {A{{\sigma }_{f}}B-A{{\sigma }_{g}}B} \right) \\
& ={\mathfrak R\left( A{{\sigma }_{f}}B \right)-\mathfrak R\left( A{{\sigma }_{g}}B \right)} \\
& \le {{{\sec }^{2}}\alpha \left( \mathfrak{R}A{{\sigma }_{f}}\mathfrak RB \right)-\left( \mathfrak{R} A{{\sigma }_{g}}\mathfrak RB \right)} \\
& ={\mathfrak RA{{\sigma }_{f}}\mathfrak RB-\mathfrak RA{{\sigma }_{g}}\mathfrak RB}+\left({{{\sec }^{2}}\alpha -1}\right)\left( \mathfrak RA{{\sigma }_{f}}\mathfrak RB \right) \\
& ={{D}_{f,g}}\left( \mathfrak RA| \mathfrak RB \right)+\left({{{\sec }^{2}}\alpha -1}\right)\left( \mathfrak RA{{\sigma }_{f}}\mathfrak RB \right),
\end{aligned}\]
which completes the proof.
\end{proof}
We give an example for Theorem \ref{thm_entropy_1}.
If we take $f(x):=\dfrac{x^t-1}{t}+1,\,\,(0<t\le 1)$ and $g(x):=1$ in Theorem \ref{thm_entropy_1}, then we have
$$
D_t(\mathfrak R A|\mathfrak R B)+(1-\sec^2\alpha)\;\mathfrak R A\le \mathfrak R\left( D_t(A|B)\right)
\le \sec^2\alpha\;D_t(\mathfrak R A|\mathfrak R B)+(\sec^2\alpha-1)\mathfrak R A.
$$
Since it is known the relation $D_t(\mathfrak R A|\mathfrak R B) \le \mathfrak R\left( D_t(A|B)\right)$ for accretive matrices $A, B$ and $0<t<1$ in \cite{rais}, the lower bound of $\mathfrak R\left( D_t(A|B)\right)$ in the inequalities above does not give a refined bound. However, we obtain the upper bound of $\mathfrak R\left( D_t(A|B)\right)$.
Finally, we have the following double inequality, which bounds $ D_{f,g}(A|B)$ between certain differences between the harmonic mean $!_t$ and the arithmetic mean $\nabla_t.$
\begin{theorem}\label{thm_entropy_2}
Let $A,B\in \Pi _{n}^{\alpha }$ and $f,g\in\mathfrak{m}$ be such that $f'(1)=g'(1)=t$. Then
\[{{{\cos }^{2}}\alpha \;\mathfrak R\left( A{{!}_{t}}B \right)-{{\sec }^{2}}\alpha \;\mathfrak R\left( A{{\nabla }_{t}}B \right)}\le \mathfrak R\left( {{D}_{f,g}}\left( A|B \right) \right)\le {{{\sec }^{2}}\alpha \;\mathfrak R\left( A{{\nabla }_{t}}B \right)-{{\cos }^{2}}\alpha \;\mathfrak R\left( A{{!}_{t}}B \right)}.\]
\end{theorem}
\begin{proof}
Noting Lemma \ref{thm_amgmhm_accretive}, we have
\[\begin{aligned}
\mathfrak R\left( {{D}_{f,g}}\left( A|B \right) \right)&={\mathfrak R\left( A{{\sigma }_{f}}B \right)-\mathfrak R\left( A{{\sigma }_{g}}B \right)} \\
& \le {{{\sec }^{2}}\alpha \;\mathfrak R\left( A{{\nabla }_{t}}B \right)-{{\cos }^{2}}\alpha \;\mathfrak R\left( A{{!}_{t}}B \right),}
\end{aligned}\]
and
\[\begin{aligned}
\mathfrak R\left( {{D}_{f,g}}\left( A|B \right) \right)&={\mathfrak R\left( A{{\sigma }_{f}}B \right)-R\left( A{{\sigma }_{g}}B \right)} \\
& \ge {{{\cos }^{2}}\alpha \;\mathfrak R\left( A{{!}_{t}}B \right)-{{\sec }^{2}}\alpha \;\mathfrak R\left( A{{\nabla }_{t}}B \right).}
\end{aligned}\]
This completes the proof.
\end{proof}
Taking $f(x):=(1-t)+tx$ and $g(x):=\{(1-t)+tx^{-1}\}^{-1}$ in Theorem \ref{thm_entropy_2}, then we obtain
$\cos^2\alpha \;\mathfrak R\left( A!_tB\right) \le \mathfrak R\left(A\nabla_tB\right)$ which is a special case of the inequality in Lemma \ref{thm_amgmhm_accretive}. If we take $f(x):=\{(1-t)+tx^{-1}\}^{-1}$ and $g(x):=(1-t)+tx$ in Theorem \ref{thm_entropy_2}, then we obtain the same inequality.
If we take $f(x):=(1-t)+tx$ and $g(x):=x^t$ in Theorem \ref{thm_entropy_2}, then we obtain
$$
(1-\sec^2\alpha)\;\mathfrak R\left(A\nabla_t B\right)+\cos^2\alpha\;\mathfrak R\left(A!_t B\right)\le
\mathfrak R\left(A\sharp_t B\right) \le (1+\sec^2\alpha)\;\mathfrak R\left(A\nabla_t B\right)-\cos^2\alpha \;\mathfrak R\left(A!_t B\right).
$$
If we take $f(x):=x^t$ and $g(x):=\{(1-t)+tx^{-1}\}^{-1}$ in Theorem \ref{thm_entropy_2}, then we obtain
$$
(1+\sec^2\alpha)\;\mathfrak R\left(A!_t B\right)-\sec^2\alpha\;\mathfrak R\left(A\nabla_t B\right)\le
\mathfrak R\left(A\sharp_t B\right) \le \sec^2\alpha \;\mathfrak R\left(A\nabla_t B\right)+(1-\cos^2\alpha)\; \mathfrak R\left(A!_t B\right).
$$
However, we find from the inequality $\cos^2\alpha \;\mathfrak R\left( A!_tB\right) \le \mathfrak R\left(A\nabla_tB\right)$ that both inequalities above do not improve the known inequality \cite{Tan2}:
$$
\cos^2\alpha\;\mathfrak R\left(A!_t B\right)\le \mathfrak R\left(A\sharp_t B\right)\le \sec^2\alpha\;\mathfrak R\left(A\nabla_t B\right).
$$
Finally, we give bounds of the weighted logarithmic mean for sectorial matrices $A,B$ using Theorem \ref{thm_entropy_2}. To this end, we review the representing function of the weighted logarithmic mean \cite{PSMA2016} given by
$$
\ell_t(x):=\frac{1-t}{t}\frac{x^t-1}{\log x}+\frac{t}{1-t}\frac{x-x^t}{\log x},\,\,(x>0,\,\,0<t<1).
$$
For $A,B>0$ and $0<t<1$, the operator version of the weighted logarithmic mean can be defined as
$$
A\ell_t B:=\frac{1-t}{t} \int_0^t A\sharp_pBdp+\int_{t}^1A\sharp_pBdp.
$$
Then we have the following corollary.
\begin{corollary}
Let $A,B\in \Pi _{n}^{\alpha }$. Then 
\begin{eqnarray*}
&& \cos^2\alpha \;\mathfrak R\left( A!_tB \right)+\mathfrak R\left( A\sharp_tB \right)-\sec^2
\alpha \;\mathfrak R\left( A\nabla_tB \right)\\
&&\le \mathfrak R\left(A\ell_tB\right)\\
&&\le \sec^2\alpha \;\mathfrak R\left( A\nabla_tB \right)+\mathfrak R\left( A\sharp_tB \right)-\cos^2\alpha \;\mathfrak R\left( A!_tB \right).
\end{eqnarray*}
\end{corollary}
\begin{proof}
We show $\left.\dfrac{\ell_t(x)}{dx}\right|_{x\to 1}=t$.
By elementary calculations, we have
$$
\frac{d}{dx}\left(\frac{x^t-1}{\log x}\right)=\frac{1-x^t+x^t\log x^t}{x(\log x)^2},\quad
\frac{d}{dx}\left(\frac{x-x^t}{\log x}\right)=\frac{x^t-x+x\log x-x^t\log x^t}{x(\log x)^2}.
$$
Applying L'Hopital's rule, we have
$$\lim_{x\to 1} \frac{1-x^t+x^t\log x^t}{x(\log x)^2}=\lim_{x\to 1}\frac{tx^{t-1}\log x^t}{2\log x+(\log x)^2}=\lim_{x\to 1} \frac{t^2x^{t-1}-t(1-t)x^{t-1}\log x^t}{2+2\log x}=\frac{t^2}{2}.$$
Since we have similarly
$$
\lim_{x\to 1}\frac{x^t-x+x\log x-x^t\log x^t}{x(\log x)^2}=\frac{1-t^2}{2},
$$
 we have
$$\left.\dfrac{\ell_t(x)}{dx}\right|_{x\to 1}=\frac{1-t}{t}\times\frac{t^2}{2}+\frac{t}{1-t}\times\frac{1-t^2}{2}=t.$$
We also show $\ell_t(x)\in \mathfrak{m}$. It is trivial $\lim\limits_{x\to 1}\ell_t(x)=1$.
We take a spectral decomposition of the bounded linear operator $A\ge 0$ as $A=\int_0^{\infty}\lambda dE_{\lambda}$. For a continuous function $f:(0,\infty)\to (0,\infty)$, we have $f(A)=\int_0^{\infty}f(\lambda)dE_{\lambda}$ by a standard functional calculus. From Fubini's theorem with $f_1(x):=\dfrac{x^t-1}{\log x}=\int_0^{t}x^pdp$, we have for any vector $u\in \mathcal{H}$
\begin{eqnarray*}
&& \langle f_1(A)u,u\rangle =\langle \int_0^{\infty}f_1(\lambda)dE_{\lambda} u,u\rangle =\langle \int_0^{\infty}\int_0^{t}\lambda^pdpdE_{\lambda} u,u\rangle\\
&&\qquad\qquad \quad=\langle \int_0^{t}\int_0^{\infty}\lambda^pdE_{\lambda}dp u,u\rangle=\int_0^t\langle A^pu,u\rangle dp.
\end{eqnarray*}
From $0<t<1$, we have $0<p<1$. Then we have $0\le A\le B \Longrightarrow f_1(A)\le f_1(B)$. Similarly we have $0\le A\le B \Longrightarrow f_2(A)\le f_2(B)$ for the function $f_2(x):=\dfrac{x-x^t}{\log x}=\int_t^1x^pdp$. Therefore $f(x)=\dfrac{1-t}{t}f_1(x)+\dfrac{t}{1-t}f_2(x)\in \mathfrak{m}$ for $0<t<1$. Thus we can apply Theorem \ref{thm_entropy_2} with $f(x):=\ell_t(x)$ and $g(x):=x^t$ and then we obtain the desired inequalities.
\end{proof}
To our knowledge, the bounds for the weighted logarithmic mean for the positive matrices case and/or scalar case have not been known yet; see \cite{FA2021,FYM2021} for example.

As we have seen, Theorem \ref{thm_entropy_1} and Theorem \ref{thm_entropy_2} give some interesting bounds using appropriate functions easily, although their estimations are not so sharp. We gave there the general forms for general functions.
\subsection*{Declarations}
\begin{itemize}
\item {\bf{Availability of data and materials}}: Not applicable
\item {\bf{Competing interests}}: The authors declare that they have no competing interests.
\item {\bf{Funding}}: This research is supported by a grant (JSPS KAKENHI, Grant Number: 21K03341) awarded to the author, S. Furuichi.
\item {\bf{Authors' contributions}}: Authors declare that they have contributed equally to this paper. All authors have read and approved this version.
\item {\bf{Acknowledgments}}: Not applicable.
\end{itemize}

\vskip 0.5 true cm

{\tiny (H. R. Moradi) Department of Mathematics, Payame Noor University (PNU), P.O. Box, 19395-4697, Tehran, Iran

	\textit{E-mail address:} hrmoradi@mshdiau.ac.ir}
	
	\vskip 0.3 true cm

{\tiny (S. Furuichi) Department of Information Science, College of Humanities and Sciences, Nihon University, Setagaya-ku, Tokyo, Japan}
{\tiny \textit{E-mail address:} furuichi.shigeru@nihon-u.ac.jp}

\vskip 0.3 true cm

{\tiny (M. Sababheh) Vice president, Princess Sumaya University for Technology, Amman, Jordan}

{\tiny\textit{E-mail address:} sababheh@psut.edu.jo}

\end{document}